\documentclass[12pt]{amsart}
\usepackage{latexsym,amsmath,amssymb,epsfig,amsthm}
\usepackage{graphicx}
\usepackage{tikz}
\usetikzlibrary{calc}
\usepackage[enableskew,vcentermath]{youngtab}
\usepackage{hyperref}
\hypersetup{colorlinks=false}
\input epsf
\newtheorem{theorem}{Theorem}[section]
\newtheorem{proposition}[theorem]{Proposition}
\newtheorem{corollary}[theorem]{Corollary}
\newtheorem{conjecture}[theorem]{Conjecture}

\newtheorem{lemma}[theorem]{Lemma}
\newtheorem{remark}[theorem]{Remark}

\newtheorem{question}[theorem]{Question}

\newcommand{\edes}{{\rm edes}}
\newcommand{\eDes}{{\rm eDes}}
\newcommand{\des}{{\rm des}}
\newcommand{\Des}{{\rm Des}}

\newcommand{\inv}{{\rm inv}}

\newcommand{\lpeak}{{\rm lpeak}}
\newcommand{\Pyr}{{\rm Pyr}}
\newcommand{\Prism}{{\rm Prism}}
\newcommand{\sd}{{\rm sd}}

\newcommand{\aA}{{\mathcal A}}
\newcommand{\bB}{{\mathcal B}}

\newcommand{\eE}{{\mathcal E}}
\newcommand{\fF}{{\mathcal F}}
\newcommand{\hH}{{\mathcal H}}

\newcommand{\lL}{{\mathcal L}}
\newcommand{\mM}{{\mathcal M}}
\newcommand{\pP}{{\mathcal P}}

\newcommand{\uU}{{\mathcal U}}
\newcommand{\zZ}{{\mathcal Z}}
\newcommand{\ba}{{\mathbf a}}
\newcommand{\bb}{{\mathbf b}}
\newcommand{\bu}{{\mathbf u}}

\newcommand{\RR}{{\mathbb R}}
\newcommand{\fS}{{\mathfrak S}}

\newcommand{\ZZ}{{\mathbb Z}}
\newcommand{\QQ}{{\mathbb Q}}
\renewcommand{\to}{\rightarrow}

\newcommand{\sm}{{\smallsetminus}}
\begin{document}
\title[Chain enumeration and real-rooted polynomials]
{Chain enumeration, partition lattices and polynomials 
with only real roots}

\author{Christos~A.~Athanasiadis}
\address{Department of Mathematics\\
National and Kapodistrian University of Athens\\
Panepistimioupolis\\ 15784 Athens, Greece}
\email{caath@math.uoa.gr}

\author{Katerina~Kalampogia-Evangelinou}
\address{Department of Mathematics\\
National and Kapodistrian University of Athens\\
Panepistimioupolis\\ 15784 Athens, Greece}
\email{kalampogia@math.uoa.gr}

\date{January 2, 2023}

\begin{abstract}
The coefficients of the chain polynomial of a finite 
poset enumerate chains in the poset by their number 
of elements. The chain polynomials of the partition 
lattices and their standard type $B$ analogues are 
shown to have only real roots. The real-rootedness 
of the chain polynomial is conjectured for all 
geometric lattices and is shown to be preserved by 
the pyramid and the prism operations on 
Cohen--Macaulay posets. As a result, new 
families of convex polytopes whose barycentric 
subdivisions have real-rooted $f$-polynomials are 
presented. An application to the face enumeration 
of the second barycentric subdivision of the 
boundary complex of the simplex is also included. 
\end{abstract}

\maketitle

\section{Introduction}
\label{sec:intro}
 
The chain polynomial of a finite poset (partially 
ordered set) $\lL$ is defined as $p_\lL(x) := 
\sum_{k \ge 0} c_k(\lL) x^k$, where $c_k(\lL)$ 
stands for the number of $k$-element chains in 
$\lL$ (thus, $p_\lL(x)$ is the $f$-polynomial of 
the order complex of $\lL$). The general question 
which motivates this paper is as follows.
\begin{question} \label{que:main}
For which finite posets does the chain polynomial 
have only real roots?  
\end{question}
This question has been studied and proven to be 
very interesting and challenging for specific classes 
of posets. For finite distributive lattices, it is 
known to be equivalent to the poset conjecture 
for natural labelings, posed in the seventies by 
Neggers~\cite{Ne78} (see also 
\cite[Conjecture~1]{Sta89}) and finally disproved 
by Stembridge~\cite{Ste07}, after counterexamples 
to a more general conjecture were found by 
Br\"and\'en~\cite{Bra04}. For face lattices of 
convex polytopes, it was raised by Brenti and 
Welker in their study of $f$-vectors of barycentric 
subdivisions~\cite{BW08}. The question is currently 
open in this case and is known to have an affirmative 
answer for face lattices of simplicial (equivalently, 
simple) polytopes~\cite{BW08}, cubical 
polytopes~\cite{Ath20b+} and, by the results 
of~\cite{Ga05}, polytopes of dimension at most five. 
Another notable result~\cite[Corollary~2.9]{Sta98} 
asserts that the chain polynomial $p_\lL(x)$ has
only real roots for every poset $\lL$ which does
not contain the disjoint union of a three-element
chain and a one-element chain as an induced 
subposet. 

This paper partly aims to show that 
Question~\ref{que:main} is very interesting for 
other classes of posets as well, especially for  
geometric lattices (the lattices of flats of matroids). 
The following statement is the main conjecture posed 
in this paper (a more precise conjecture appears in 
Section~\ref{sec:gg}).
\begin{conjecture} 
The chain polynomial $p_\lL(x)$ has only real roots 
for every geometric lattice $\lL$. 
\label{conj:main}
\end{conjecture}

Our first main result verifies this conjecture for
some important geometric lattices, such as the subspace 
lattice $\lL_n(q)$ of all linear subspaces of an 
$n$-dimensional vector space over the field with $q$ 
elements, the partition lattice $\Pi_n$ 
\cite[Section~3.1]{StaEC1} and its standard type $B$ 
analogue $\Pi^B_n$ \cite[Section~1.3]{Wa07}. The 
statement about uniform matroids follows from the 
main result of~\cite{BW08} and is included for the 
sake of completeness. 
\begin{theorem} \label{thm:mainA} 
Conjecture~\ref{conj:main} is true for
\begin{itemize}
\itemsep=0pt
\item[(a)]
the subspace lattices $\lL_n(q)$,

\item[(b)]
the partition lattices $\Pi_n$ and $\Pi^B_n$,

\item[(c)]
the lattices of flats of near-pencils and uniform 
matroids.
\end{itemize}
\end{theorem}

Conjecture~\ref{conj:main} has also been verified 
computationally for all geometric lattices with at 
most nine atoms.

Our second main result gives constructions
of posets which preserve the real-rootedness of the 
chain polynomial. More specifically, there are 
natural operations $\Pyr$ and $\Prism$ on posets
(called pyramid and prism, see Section~\ref{sec:pre}), 
such that if $\lL$ is the face lattice of a convex 
polytope $\pP$, then 
$\Pyr(\lL)$ and $\Prism(\lL)$ are the face lattices 
of the pyramid and the prism over $\pP$, 
respectively. Among other applications, the 
following statement implies the existence of large 
families of nonsimplicial, nonsimple and noncubical 
polytopes in any dimension, the face lattices of 
which have real-rooted chain polynomials.

\begin{theorem} \label{thm:mainB} 
If the chain polynomial of a bounded Cohen--Macaulay 
poset $\lL$ has only real roots, then so do the chain
polynomials of the pyramid and the prism over $\lL$. 
\end{theorem}

The content, methods and structure of this paper 
may be described as follows. Section~\ref{sec:pre}
provides basic definitions and useful background 
from algebraic, enumerative and geometric 
combinatorics (mainly on chain enumeration in 
posets) and the theory of real-rooted polynomials.
Sections~\ref{sec:subspace} and~\ref{sec:Pns} prove 
parts (a) and (b) of Theorem~\ref{thm:mainA}, 
respectively. The proofs depend on specific 
combinatorial features of these posets and do not 
seem to extend easily to other geometric lattices. 
They proceed by exploiting explicit combinatorial 
interpretations of the flag $h$-vectors and the 
$h$-polynomials of the order complexes of the posets 
in question. To the best of our knowledge, such 
combinatorial interpretations have not 
appeared in the literature before for the partition 
lattices $\Pi_n$ and $\Pi^B_n$. Part (a) of 
Theorem~\ref{thm:mainA} is easier to prove and 
serves as an introductory example.

Section~\ref{sec:gg} discusses Question~\ref{que:main} 
for some general classes of Cohen--Macaulay posets, 
including that of geometric lattices, proves 
Theorem~\ref{thm:mainB} and deduces from that part 
(c) of Theorem~\ref{thm:mainA} (see 
Proposition~\ref{prop:pencil-uniform}). 
Proposition~\ref{prop:sd(Z)} suggests that the 
real-rootedness of chain polynomials of geometric 
lattices should perhaps be studied in connection to 
that of chain polynomials of face lattices of 
zonotopes and oriented matroids. The proofs in 
Section~\ref{sec:gg} rely heavily on results of 
Ehrenborg and~Readdy~\cite{ER98} on the chain 
enumeration of pyramids and prisms and of Billera, 
Ehrenborg and~Readdy~\cite{BER97} on the chain 
enumeration of big face lattices of oriented 
matroids (reviewed in Section~\ref{sec:pre}). 
Section~\ref{sec:bary} applies 
Proposition~\ref{prop:sd(Z)} to give an unexpected 
combinatorial interpretation of the $h$-polynomial 
of the second barycentric subdivision of the 
boundary complex of a simplex 
and of its associated $\gamma$-polynomial, thus 
solving a problem posed in~\cite{Ath18}. 

As noted already, the chain polynomial $p_\lL(x)$ 
coincides with the $f$-polynomial of the order 
complex $\Delta(\lL)$ of a poset $\lL$. The 
results of Sections~\ref{sec:subspace}, \ref{sec:Pns} 
and~\ref{sec:gg} are phrased in terms of the 
$h$-polynomial of $\Delta(\lL)$ instead, which is 
more natural from an algebraic combinatorics point 
of view for the classes of posets we are interested 
in.

\section{Preliminaries} 
\label{sec:pre}

This section includes preliminaries on notation 
and definitions and discusses some of the key tools 
and results from chain enumeration and the theory of 
real-rooted polynomials which will be used in this 
paper. We assume familiarity with main objects of 
study in algebraic, enumerative and geometric 
combinatorics, such as posets, matroids, simplicial 
complexes, convex polytopes and oriented matroids; 
standard references are \cite{Bj92, OM, Ox, StaCCA, 
StaEC1, Wa07, Zi95}. Any undefined terminology can 
be found there. 

We will denote by $\fS_n$ the 
symmetric group of permutations of $[n] := 
\{1, 2,\dots,n\}$ and by $|S|$ the cardinality of 
a finite set $S$. 

\subsection{Simplicial complexes and face enumeration}

Given an $(n-1)$-dimensional (finite, abstract) 
simplicial complex $\Delta$, the 
\emph{$f$-polynomial} and the \emph{$h$-polynomial} 
are defined as
\begin{eqnarray*}
f(\Delta, x) & = & \sum_{i=0}^n f_{i-1}(\Delta) x^i 
\\ h(\Delta, x) & = & 
(1-x)^n f(\Delta, \frac{x}{1-x}) \ = \ 
\sum_{i=0}^n f_{i-1} (\Delta) \, x^i (1-x)^{n-i} ,  
\end{eqnarray*}
where $f_{i-1}(\Delta)$ stands for the number of 
$(i-1)$-dimensional faces of $\Delta$. The 
$h$-polynomial has nonnegative coefficients for all
Cohen--Macaulay simplicial complexes 
\cite[Section~II.3]{StaCCA} (this 
property is meant to be considered here over the 
field $\QQ$ of rational numbers) and, in particular, 
for all simplicial complexes of interest in this 
paper. Morever, $h(\Delta, x)$ can be considered as 
an $x$-analogue of $f_{n-1}(\Delta)$, to which its 
coefficients sum up. 
Since $f(\Delta, x)$ has only real roots if 
and only if so does $h(\Delta, x)$, the former can 
be replaced by the latter as far as
real-rootedness is concerned.

\subsection{Order complexes and chain enumeration}

We will be mostly interested in order complexes of
posets. The \emph{order complex} of a finite poset 
$(\lL, \preceq)$ is defined \cite[Section~3.8]{StaEC1} 
as the simplicial complex $\Delta(\lL)$ which consists 
of all chains in $\lL$. We have $f(\Delta(\lL), x) 
= p_\lL(x)$, where $p_\lL(x)$ is the chain polynomial 
defined in the introduction. To simplify notation, 
we set $h_\lL(x) = h(\Delta(\lL), x)$ throughout this 
paper. Thus, $p_\lL(x)$ and $h_\lL(x)$ are related 
by the equation
\[ h_\lL(x) \ = \ (1-x)^n \, p_\lL \left( \frac{x}{1-x}
   \right) , \]
where $n$ is the largest cardinality of a 
chain in $\lL$. The polynomial $h_\lL(x)$ is an  
$x$-analogue of the number of $n$-element chains of 
$\lL$, which are exactly the $(n-1)$-dimensional faces 
of $\Delta(\lL)$. 
  
Suppose now that $\lL$ has a minimum element $\hat{0}$,
a maximum element $\hat{1}$ and that it is graded of 
rank $n$, say with rank function $\rho: P \to \{0, 
1,\dots,n\}$. Thus, all maximal chains in $\lL$ have 
exactly $n+1$ elements. Following 
\cite[Section~3.13]{StaEC1}, for $S \subseteq [n-1]$ we 
denote by $\alpha_\lL(S)$ the number of maximal chains 
of the subposet $\{ t \in \lL: \rho(t) \in S \} \cup 
\{\hat{0}, \hat{1}\}$ of $\lL$ and set 
\[ \beta_\lL(S) \ = \ \sum_{T \subseteq S} \, 
   (-1)^{|S-T|} \, \alpha_\lL(T). \]
The numbers $\beta_\lL(S)$ for $S \subseteq [n-1]$ 
are the entries of the \emph{flag $h$-vector} of $\lL$. 
They are nonnegative if $\lL$ is Cohen--Macaulay (see, 
for instance, \cite[Theorem~4.4]{StaCCA}) and 
refine the coefficients of $h_\lL(x)$, in the sense 
that
\begin{equation} \label{eq:hL(x)}
h_\lL(x) \ = \sum_{S \subseteq [n-1]} \beta_\lL(S) 
x^{|S|}.
\end{equation}
We also set $\overline{\lL} = \lL \sm \{ \hat{0}, 
\hat{1} \}$ and recall that $h_\lL(x) = 
h_{\lL \sm \{ \hat{0} \}}(x) = 
h_{\lL \sm \{ \hat{1} \}}(x) = h_{\overline{\lL}}(x)$. 

The theory of edge labelings can provide useful 
combinatorial interpretations of the numbers 
$\beta_\lL(S)$. We denote by $\eE(\lL)$ and 
$\mM(\lL)$ the set of covering relations and maximal 
chains of $\lL$, respectively. An edge labeling of 
$\lL$ for us will be a map $\lambda: \eE(\lL) \to 
\ZZ$. Given a maximal chain $C: \hat{0} = c_0 \prec 
c_1 \prec \cdots \prec c_n = \hat{1}$ of $\lL$, such 
a map induces the sequence of labels 
\[ \lambda(C) \ = \ (\lambda(c_0,c_1), 
   \lambda(c_1,c_2),\dots,\lambda(c_{n-1},c_n)) . \]

\smallskip
\noindent
We denote by $\Des_\lambda(C)$ the set of indices $i 
\in [n-1]$ for which $\lambda(c_{i-1},c_i) \ge 
\lambda(c_i,c_{i+1})$ and call $C$ \emph{strictly 
increasing} with respect to $\lambda$ if no such 
index exists. By restricting $\lambda$, all these 
definitions apply to the closed intervals in $\lL$. 
We say that $\lambda$ is a \emph{strict $R$-labeling} 
if every closed interval in $\lL$ has a unique strictly
increasing maximal chain with respect to $\lambda$. 
Under this assumption on $\lambda$, for every $S 
\subseteq [n-1]$, $\beta_\lL(S)$ is equal to the 
number of maximal chains $C \in \mM(\lL)$ such that
$\Des_\lambda(C) = S$; see \cite[Section~3.14]{StaEC1} 
\cite[Section~3.2]{Wa07} for examples and more 
information on $R$-labelings.

The flag $h$-vector of $\lL$ is nicely encoded by the 
$\ba \bb$-index $\Psi_\lL$. Given noncommuting 
variables $\ba$ and $\bb$, this may be defined as
\[ \Psi_\lL \ = \ \Psi_\lL(\ba,\bb) \ = 
   \sum_{S \subseteq [n-1]} \beta_\lL(S) \bu_S , \]
where $\bu_S = u_1 u_2 \cdots u_{n-1}$, with 
\[ u_i \ = \ \begin{cases}
   \bb, & \text{if $i \in S$} \\
   \ba, & \text{if $i \not\in S$}, \end{cases} \]
is the $\ba \bb$-monomial associated to $S$ in the 
standard way. We note that Equation~(\ref{eq:hL(x)}) 
may be rewritten as $h_\lL(x) = \Psi_\lL(1,x)$.

\subsection{Face lattices of polytopes}

We will denote by $\fF(\pP)$ the face lattice of a 
polytope $\pP$. The \emph{barycentric subdivisions} 
$\sd(\pP)$ and $\sd(\partial \pP)$ of $\pP$ and its
boundary complex $\partial \pP$ are defined as the 
order complexes $\Delta(\lL \sm \{ \hat{0} \})$ 
and $\Delta(\overline{\lL})$, respectively, where 
$\lL = \fF(\pP)$. In particular, the $f$-polynomials 
of $\sd(\pP)$ and $\sd(\partial \pP)$ are the chain 
polynomials of $\lL \sm \{ \hat{0} \}$ and 
$\overline{\lL}$ and their real-rootedness is 
equivalent to that of $h_\lL(x) 
= h_{\lL \sm \{ \hat{0} \}}(x) = h(\sd(\pP),x) = 
h_{\overline{\lL}} (x) = h(\sd(\partial \pP),x)$.

There is another lattice associated 
to any zonotope $\zZ$, namely the geometric lattice 
of flats of the matroid defined by the generators 
of $\zZ$. This lattice, say $\lL(\zZ)$, is isomorphic 
to the intersection lattice of the linear hyperplane 
arrangement $\hH_\zZ$ corresponding to $\zZ$; the face
poset $\fF(\zZ) \sm \{\hat{0}\}$ is anti-isomorphic 
to the face poset of $\hH_\zZ$ \cite[Section~7.3]{Zi95}. 
The main result of~\cite{BER97} expresses the flag 
$h$-vector of $\fF(\zZ)$ in terms of that of $\lL(\zZ)$.
Consider the linear function $\omega: \ZZ \langle 
\ba, \bb \rangle \to \ZZ \langle \ba, \bb \rangle$ 
defined as follows: if $v$ is any $\ba \bb$-word, 
then $\omega(v)$ is obtained from $v$ by first 
replacing each occurrence of $\ba \bb$ with 
$2(\ba \bb + \bb \ba)$ and then replacing each of 
the remaining letters with $\ba + \bb$. 

\begin{theorem} {\rm (\cite[Corollary~3.2]{BER97})}
We have $\Psi_{\fF(\zZ)} = \omega (\ba \cdot 
\Psi_\lL)$ for every zonotope $\zZ$, where $\lL$ is
the lattice of flats of the matroid associated to $\zZ$. 
\label{thm:BER} 
\end{theorem}

Given a bounded, finite poset $\lL$, 
the \emph{pyramid} over $\lL$ is defined as $\Pyr(\lL) 
= \lL \times \lL_1$, where $\lL_1$ is the 2-element 
chain. The \emph{prism} over $\lL$, denoted 
$\Prism(\lL)$, is defined as the poset obtained by 
adding a minimum element to $(\lL \sm \{\hat{0}\}) 
\times (\lL_2 \sm \{\hat{0}\})$, where $\lL_2 = \lL_1 
\times \lL_1$. Then, for every polytope $\pP$, the 
posets $\Pyr(\fF(\pP))$ and $\Prism(\fF(\pP))$ are 
isomorphic to the face lattices of the pyramid and 
the prism over $\pP$, respectively. 
Following~\cite{ER98}, we consider the linear 
derivation $D: \ZZ \langle \ba, \bb \rangle \to \ZZ 
\langle \ba, \bb \rangle$ defined by setting $D(\ba) 
= D(\bb) = \ba \bb + \bb \ba$.

\begin{theorem} {\rm (\cite[Theorem~4.4]{ER98})}
We have
\begin{eqnarray}
2 \, \Psi_{\Pyr(\lL)} & = & \Psi_\lL \cdot (\ba + \bb) 
  \, + \, (\ba + \bb) \cdot \Psi_\lL \, + \, D(\Psi_\lL) 
	, \label{eq:ER1} \\ 
\Psi_{\Prism(\lL)} & = & \Psi_\lL \cdot (\ba + \bb) \, +
  \, D(\Psi_\lL) \label{eq:ER2}
\end{eqnarray}
for every bounded, graded finite poset $\lL$.
\label{thm:ER} 
\end{theorem}

\subsection{Real-rooted polynomials}
A polynomial $f(x)$ with real coefficients is called
\emph{real-rooted} if every root of $f(x)$ is real, or 
$f(x) \equiv 0$. 

A real-rooted polynomial $f(x)$, with 
roots $\alpha_1 \ge \alpha_2 \ge \cdots$, is said to 
\emph{interlace} a real-rooted polynomial $g(x)$, with 
roots $\beta_1 \ge \beta_2 \ge \cdots$, if
\[ \cdots \le \alpha_2 \le \beta_2 \le \alpha_1 \le
   \beta_1. \]
By convention, the zero polynomial interlaces and is 
interlaced by every real-rooted polynomial. 
A sequence $(f_0(x), f_1(x),\dots,f_m(x))$ of 
real-rooted polynomials is called \emph{interlacing} 
if $f_i(x)$ interlaces $f_j(x)$ for $0 \le i < j \le 
m$. The following standard lemma (see, for instance,
\cite[Section~7.8]{Bra15}) will be applied 
several times in this paper.
\begin{lemma} \label{lem:int-seq} 
Let $(f_0(x), f_1(x),\dots,f_m(x))$ be an interlacing 
sequence of real-rooted polynomials with positive 
leading coefficients.
\begin{itemize}
\itemsep=0pt
\item[(a)]
Every nonnegative linear combination $f(x)$ of 
$f_0(x), f_1(x),\dots,f_m(x)$ is real-rooted. Moroever, 
$f(x)$ interlaces $f_m(x)$ and it is interlaced by 
$f_0(x)$. 

\item[(b)]
The sequence $(g_0(x), g_1(x),\dots,g_{m+1}(x))$ defined
by
\[ g_k(x) \ = \ x \sum_{i=0}^{k-1} f_i(x) \, + \, 
             \sum_{i=k}^m f_i(x) \]
for $k \in \{0, 1,\dots,m+1\}$ is also interlacing.
\end{itemize}
\end{lemma}

A polynomial $f(x) = a_0 + a_1 x + \cdots + a_n x^n$ 
with real coefficients is called \emph{symmetric}, 
with center of symmetry $n/2$, if $a_i = a_{n-i}$ for 
all $0 \le i \le n$. Then, $f(x) = 
\sum_{i=0}^{\lfloor n/2 \rfloor} \gamma_i x^i 
(1+x)^{n-2i}$ for some uniquely defined real numbers 
$\gamma_0, \gamma_1,\dots,\gamma_{\lfloor n/2 \rfloor}$ 
and $\gamma(x) = \sum_{i=0}^{\lfloor n/2 \rfloor} 
\gamma_i x^i$ is the \emph{$\gamma$-polynomial} 
associated to $f(x)$. The latter is called 
\emph{$\gamma$-positive} if $\gamma_i \ge 0$ for every 
$i$; see \cite{Ath18} \cite[Section~7.3]{Bra15} for 
more information on this concept.  

\section{Subspace lattices} 
\label{sec:subspace}

This section confirms Theorem~\ref{thm:mainA} for 
subspace lattices. The proof essentially follows by
combining \cite[Theorem~3.13.3]{StaEC1} with
\cite[Theorem~5.4]{SV15} but serves as a paradigm 
for the proofs of the corresponding statements for 
the lattices $\Pi_n$ and $\Pi^B_n$ in the following 
section, which are more involved.

Given a positive integer $n$ and a prime power $q$,
let $V_n(q)$ be an $n$-dimensional vector space 
over the field with $q$ elements. The \emph{subset 
lattice} $\lL_n$, known as the Boolean algebra of 
rank $n$, and the \emph{subspace lattice} $\lL_n(q)$ 
are defined \cite[Example~3.1.1]{StaEC1} as the 
set of all subsets of the set $[n]$ and as the set 
of all linear subspaces of $V_n(q)$, respectively, 
partially ordered by inclusion. The posets $\lL_n$ 
and $\lL_n(q)$ are geometric lattices of rank $n$, 
the latter being considered as a $q$-analogue of the 
former. We recall that a \emph{descent} of a permutation 
$w \in \fS_n$ is an index $i \in [n-1]$ such that $w(i) 
> w(i+1)$ and denote by $\des(w)$ and $\Des(w)$ the  
number and the set of all descents of $w$, 
respectively. 

The order complex $\Delta(\overline{\lL}_n)$ is 
isomorphic to the first barycentric subdivision of 
the boundary complex of the $(n-1)$-dimensional 
simplex and its $h$-polynomial $h_{\lL_n}(x) =
h(\Delta(\overline{\lL}_n),x)$ is known 
\cite[Theorem~9.1]{Pet15} to be equal to the $n$th 
Eulerian polynomial
\begin{equation*} \label{eq:defAn}
A_n(x) \ := \, \sum_{w \in \fS_n} x^{\des(w)} . 
\end{equation*}
The latter is well known \cite[Section~7.8.1]{Bra15}
to have only real roots 
for every $n$. As already discussed, the following 
statement is equivalent to part (a) of 
Theorem~\ref{thm:mainA}.
\begin{proposition} \label{prop:subspace}
The polynomial $h_{\lL}(x)$ is real-rooted for every 
subspace lattice $\lL = \lL_n(q)$.
\end{proposition}

\begin{proof}
Setting $\lL := \lL_n(q)$, we have 
\cite[Theorem~3.13.3]{StaEC1} 
\[ \beta_\lL(S) \ = \sum_{w \in \fS_n : \, \Des(w) = 
   S} q^{\inv(w)} , \]
where $\inv(w)$ is the number of pairs of indices 
$1 \le i < j \le n$ for which $w(i) > w(j)$. This 
formula and Equation~(\ref{eq:hL(x)}) imply that 
$h_{\lL}(x) = A_n(x;q)$, where 
\[ A_n(x;q) \ := \ \sum_{w \in \fS_n} q^{\inv(w)} 
   x^{\des(w)} \]
is one of the well studied $q$-analogues of $A_n(x)$
\cite[Section~3.19]{StaEC1}. The fact that $A_n(x;q)$ 
is real-rooted for every positive real number $q$ was 
shown in \cite[Theorem~5.4]{SV15}. To give a 
self-contained proof, we set 
\[ A_{n,k}(x;q) \ = \sum_{w \in \fS_n: \, w(n) = k} 
   q^{\inv(w)} x^{\des(w)} \]
for $k \in [n]$ and note that $A_n(x;q) = 
A_{n+1,n+1}(x;q)$ for every $n$ and that 
\begin{equation} \label{eq:Anq-rec}
 A_{n+1,k}(x;q) \ = \ q^{n+1-k} \left( \,
 \sum_{i=1}^{k-1} A_{n,i}(x;q) \, + \, x 
\sum_{i=k}^n A_{n,i}(x;q) \right)
\end{equation}
for $k \in [n+1]$. An application of 
Lemma~\ref{lem:int-seq} shows by induction on $n$ 
that the sequence $(A_{n,n}(x;q), 
A_{n,n-1}(x;q),\dots,A_{n,1}(x;q))$ is interlacing 
for all $n$ and positive $q$. As a result, 
$A_n(x;q) = A_{n+1,n+1}(x;q)$ is real-rooted for 
all $n$ and positive $q$.
\end{proof}

The previous argument also shows that $A_n(x;q)$ 
interlaces $A_{n+1}(x;q)$ for all $n$ and all 
positive $q$. This follows from part (a) of 
Lemma~\ref{lem:int-seq} since, by~(\ref{eq:Anq-rec}), 
$A_n(x;q) = A_{n+1,n+1}(x;q)$ is a positive linear 
combination of the 
$A_{n,k}(x;q)$ for $1 \le k \le n$ and $A_{n,n}(x;q) 
= A_{n-1}(x;q)$. For the real-rootedness of other 
$q$-analogues of $A_n(x)$, see \cite[Section~5]{SV15}.

\section{Partition lattices} 
\label{sec:Pns}

This section proves part (b) of 
Theorem~\ref{thm:mainA}. As in
Section~\ref{sec:subspace}, we will show the 
equivalent statement that $h_{\Pi_n}(x)$ and 
$h_{\Pi^B_n}(x)$ are real-rooted for every 
$n \ge 1$. 
			
\subsection{The partition lattice.}
\label{sec:Pn}

We recall that $\Pi_n$ consists of all partitions 
of the set $[n]$, partially ordered by reverse 
refinement. It is isomorphic to the intersection 
lattice of the Coxeter hyperplane arrangement of 
type $A_{n-1}$ and, as such, it is a geometric 
lattice of rank $n-1$. We will first give an 
explicit combinatorial interpretation of the flag 
$h$-vector of $\Pi_n$
and will deduce one for $h_{\Pi_n}(x)$. A 
recurrence for the entries of the former was found 
by Sundaram~\cite[Proposition~2.16]{Su94}; an 
additional formula appears as 
\cite[Proposition~2.18]{Su94}. We consider the 
multiset
\[ \aA_n \ := \ \{1\} \times \{1, 1, 2\} \times 
   \{1, 1, 1, 2, 2, 3\} \times \cdots \times 
	 \{1, 1,\dots,1,\dots,n-2, n-2, n-1\} , \]
e.g., $\aA_3 = \{ (1,1), (1,1), (1,2) \}$.
We define the \emph{descent set} of $\sigma  = 
(\sigma_1, \sigma_2,\dots,\sigma_{n-1}) \in \aA_n$ 
as $\Des(\sigma) = \{ i \in [n-2]: \sigma_i
\ge \sigma_{i+1} \}$ and denote its cardinality by
$\des(\sigma)$. The multiset $\aA_n$ has 
$\prod_{k=2}^n {k \choose 2} = n! \, (n-1)! / 2^{n-1}$ 
elements, as many as the maximal chains of $\Pi_n$. 
Moreover, one element of $\aA_n$ has empty descent 
set and $(n-1)!$ of them have descent set equal to 
$[n-2]$. These facts agree with the following 
statement.

\begin{proposition} \label{prop:betaPi} 
For every $n \ge 2$ and every $S \subseteq [n-2]$,
the number $\beta_{\Pi_n}(S)$ is equal to the 
number of elements of the multiset $\aA_n$ with 
descent set equal to $n-1-S := \{ n-1-x: x \in S\}$. 
In particular,
\begin{equation} \label{eq:hPi(x)}
h_{\Pi_n}(x) \ = \, \sum_{\sigma \in \aA_n} 
x^{\des(\sigma)} 
\end{equation}
for every $n \ge 2$.
\end{proposition}

To prepare for the proof, we recall 
\cite[Section~3.2.2]{Wa07} the following edge 
labeling of $\Pi_n$, due to Gessel. For a covering
relation $(x, y) \in \eE(\Pi_n)$ we define $\lambda
(x, y)$ as the maximum of $\min(B)$ and $\min(B')$,
where $y$ is obtained from $x$ by merging the blocks
$B$ and $B'$ of $x$. The labeling $\lambda: \eE(\Pi_n)
\to \{2, 3,\dots,n\}$ is a strict $R$-labeling (even
a strict EL-labeling). In particular, $\beta_{\Pi_n}
(S)$ is equal to the number of maximal chains $C$ of 
$\Pi_n$ such that $\Des_\lambda(C) = S$, 
for every $S \subseteq [n-2]$.

\medskip
\noindent
\emph{Proof of Proposition~\ref{prop:betaPi}}.
We consider the multiset
\[ \aA^\ast_n \ := \ 
   \{2, 3, 3, 4, 4, 4,\dots,n, n,\dots,n\} 
	 \times \cdots \times \{2, 3, 3, 4, 4, 4\} 
	 \times \{2, 3, 3\} \times \{2\} \]
and set $\Des^\ast(\sigma) = \{ i \in [n-2]: 
\sigma_i > \sigma_{i+1} \}$ for $\sigma  = 
(\sigma_1, \sigma_2,\dots,\sigma_{n-1}) \in 
\aA^\ast_n$. The bijection $\aA_n \mapsto 
\aA^\ast_n$ defined by 
\[ (\sigma_1, \sigma_2,\dots,\sigma_{n-1}) \ 
   \mapsto \ (n+1-\sigma_{n-1}, 
	 n-\sigma_{n-2},\dots,3-\sigma_1) \]
shows that for every $S \subseteq [n-2]$, the 
number of elements $\sigma \in \aA^\ast_n$ with 
$\Des^\ast(\sigma) = S$ is equal to the number of 
elements $\sigma \in \aA_n$ with $\Des(\sigma) = 
n-1-S$. Thus, we need to show that the former is 
equal to $\beta_{\Pi_n}(S)$.

For that, we employ Gessel's labeling $\lambda: 
\eE(\Pi_n) \to \{2, 3,\dots,n\}$. Given a covering
relation $(x, y) \in \eE(\Pi_n)$, we list the 
blocks of $x = \{B_1, B_2,\dots,B_k\}$ so that 
$\min(B_1) < \min(B_2) < \cdots < \min(B_k)$ and 
set $\varphi(x,y) = (i,j)$ and $\psi(x,y) = j$, 
where $i < j$ and $y$ is obtained from $x$ by 
merging $B_i$ with $B_j$. For a maximal chain 
$C: \hat{0} = c_1 \prec c_2 \prec \cdots \prec
c_n = \hat{1}$ of $\Pi_n$ we set 
\begin{eqnarray*}
\widetilde{\varphi}(C) & = & (\varphi(c_1,c_2), 
\varphi(c_2,c_3),\dots,\varphi(c_{n-1},c_n)) , \\ 
\widetilde{\psi}(C) & = & (\psi(c_1,c_2), 
          \psi(c_2,c_3),\dots,\psi(c_{n-1},c_n))
\end{eqnarray*}
and consider the resulting maps 
\begin{eqnarray*}
\widetilde{\varphi} : & \mM(\Pi_n) \ \to & 
{[n] \choose 2} \times {[n-1] \choose 2} \times 
\cdots \times {[2] \choose 2} , \\ 
\widetilde{\psi} : & \mM(\Pi_n) \ \to & 
\{2, 3,\dots,n\} \times \{2, 3,\dots,n-1\} \times 
\cdots \times \{2\} ,
\end{eqnarray*}

\smallskip
\noindent
where ${T \choose 2}$ stands for the set of 
2-element subsets of $T$.

Clearly, $\widetilde{\varphi}$ is a bijection. We 
claim that $\Des_\lambda(C) = \Des^\ast
(\widetilde{\psi}(C))$ for every $C \in \mM(\Pi_n)$. 
This would imply that $\beta_{\Pi_n}(S)$, which is 
equal to the number of chains $C \in \mM(\Pi_n)$ 
with $\Des_\lambda(C) = S$, is also equal to the 
number of chains $C \in \mM(\Pi_n)$ with 
$\Des^\ast(\widetilde{\psi}(C)) = S$. Since 
$\widetilde{\varphi}$ is a bijection, the latter
is equal to the number of $\sigma \in \aA^\ast_n$ 
with $\Des^\ast(\sigma) = S$.

Thus, it remains to verify the claim. Equivalently,
for $(x, y), (y, z) \in \eE
(\Pi_n)$, we need to verify that $\lambda(x, y) \ge 
\lambda(y, z) \Leftrightarrow \psi(x, y) > \psi
(y, z)$. Indeed, let $x = \{B_1, B_2,\dots,B_k\}$ 
with $\min(B_1) < \min(B_2) < \cdots < \min(B_k)$ 
and suppose that $y$ is obtained from $x$ by 
merging $B_i$ with $B_j$, where $i < j$. Then, 
$\lambda(x, y) = \min(B_j)$ and $\psi(x, y) = j$.
Each one of the inequalities $\lambda(y, z) \le 
\min(B_j)$ and $\psi(y, z) < j$ is equivalent to 
the statement that $z$ is obtained from $y$ by 
merging two blocks other than $B_{j+1},\dots,B_k$ 
and the proof follows. 
\qed

\medskip
For small values of $n$,
\[  h_{\Pi_n}(x) \ = \ \begin{cases}
    1, & \text{if $n = 1$} \\
    1, & \text{if $n = 2$} \\
    1 + 2x, & \text{if $n=3$} \\
    1 + 11x + 6x^2, & \text{if $n=4$} \\
    1 + 47x + 108x^2 + 24x^3, & 
		                  \text{if $n=5$} \\
    1 + 197x + 1268x^2 + 1114x^3 + 120x^4, & 
		                  \text{if $n=6$} \\
    1 + 870x + 13184x^2 + 29383x^3 + 12542x^4 
		         + 720x^5, & \text{if $n=7$}. \\
    \end{cases}  \]

The following corollary proves and strengthens
Theorem~\ref{thm:mainA}.
\begin{corollary} \label{cor:hPi(x)}
The polynomial $h_{\Pi_n}(x)$ is real-rooted and 
it interlaces $h_{\Pi_{n+1}}(x)$ for every $n \ge 1$.
\end{corollary}

\begin{proof}
For $k \in [n-1]$ we set 
\[ h_{n,k}(x) \ = \sum_{\sigma \in \aA_{n,k}} 
   x^{\des(\sigma)} , \]
where $\aA_{n,k}$ is the multiset consisting of 
all words $(\sigma_1, \sigma_2,\dots,\sigma_{n-1})  
\in \aA_n$ with $\sigma_{n-1} = k$. 
By Proposition~\ref{prop:betaPi}, 
\[ h_{\Pi_n}(x) \ = \ h_{n+1,n}(x) \ = \ 
                      \sum_{k=1}^{n-1} h_{n,k}(x) \]
for every $n \ge 2$ and 
\[ h_{n+1,k}(x) \ = \ (n+1-k) \left( \,
   \sum_{i=1}^{k-1} h_{n,i}(x) \, + \, x 
	 \sum_{i=k}^{n-1} h_{n,i}(x) \right) \]
for $k \in [n]$. Since $(h_{n+1,n}(x), 
h_{n+1,n-1}(x),\dots,h_{n+1,1}(x))$ is an interlacing 
sequence if and only if $(h_{n+1,n}(x), \frac{1}{2} 
h_{n+1,n-1}(x),\dots,\frac{1}{n} h_{n+1,1}(x))$ 
has the same property, an application of 
Lemma~\ref{lem:int-seq} shows by induction on $n$ that 
$(h_{n,n-1}(x), h_{n,n-2}(x),\dots,h_{n,1}(x))$ is 
interlacing for every $n \ge 2$ and that $h_{\Pi_n}
(x) = \sum_{k=1}^{n-1} h_{n,k}(x)$ is interlaced by 
$h_{n,n-1}(x) = h_{\Pi_{n-1}}(x)$.
\end{proof}


The following conjecture has been verified for 
$n \le 20$.
\begin{conjecture} \label{conj:hPi(x)}
The polynomial $h_{\Pi_n}(x)$ is interlaced by 
the Eulerian polynomial $A_{n-1}(x)$ for every 
$n \ge 2$.
\end{conjecture}
	
\subsection{The partition lattice of type $B$.}
\label{sec:PBn}

We find it convenient to define $\Pi^B_n$ as the set 
of all partitions $\pi$ of $\{-n,-n+1,\dots,n\}$ with 
the following properties: 
\begin{itemize}
\itemsep=0pt
\item[{\rm (i)}]
$B \in \pi \Rightarrow (-B) \in \pi$, 

\item[{\rm (ii)}]
if $\{i,-i\} \subseteq B$ for some $i \in [n]$ and 
some block $B \in \pi$, then $0 \in B$.
\end{itemize}
For example, $\{ \{0, 2, -2\}, \{1, -3, 5\}, \{-1, 3, -5\}, 
\{4\}, \{-4\} \} \in \Pi^B_5$. The partial order on 
$\Pi^B_n$ is again reverse refinement. 
The unique block of $\pi \in \Pi^B_n$ containing 
0 is called the \emph{zero block}. The poset $\Pi^B_n$ is 
isomorphic to the intersection lattice of the Coxeter 
hyperplane arrangement of type $B_n$ and therefore it 
is a geometric lattice of rank $n$. 

The proof of Theorem~\ref{thm:mainA} for the lattice
$\Pi^B_n$ parallels that for $\Pi_n$. One can easily 
verify that $\Pi^B_n$ has exactly $(n!)^2$ maximal 
chains. We consider the multiset
\[ \bB_n \ := \ \{1\} \times \{1, 1, 1, 2\} \times 
   \{1, 1, 1, 1, 1, 2, 2, 2, 3\} \times \cdots \times 
	 \{1, 1,\dots,1,\dots,n-1, n-1, n-1, n\} , \]
where the $k$th factor has $2k-2i+1$ elements equal 
to $i$, for $1 \le i \le k$. We define the descent 
set $\Des(\sigma) \subseteq [n-1]$ and its cardinality 
$\des(\sigma)$ for $\sigma \in \bB_n$ just as for 
elements of $\aA_n$. We note that $\bB_n$ 
has $(n!)^2$ elements, one of which has empty descent 
set and $(2n-1)!!$ of which have descent set equal to 
$[n-1]$. These facts agree with the following 
analogue of Proposition~\ref{prop:betaPi}.

\begin{proposition} \label{prop:betaPiB} 
For every $n \ge 1$ and every $S \subseteq [n-1]$,
the number $\beta_{\Pi^B_n}(S)$ is equal to the 
number of elements of the multiset $\bB_n$ with 
descent set equal to $n-S$. In particular,
\begin{equation} \label{eq:hPiB(x)}
h_{\Pi^B_n}(x) \ = \, \sum_{\sigma \in \bB_n} 
x^{\des(\sigma)} 
\end{equation}
for every $n \ge 1$.
\end{proposition}
		
The proof uses the following analogue of Gessel's
edge labeling for $\Pi_n$. Let us denote by $|B|$ 
the set of absolute values of the elements of a set
$B \subseteq \ZZ$. Given a covering relation 
$(x, y) \in \eE(\Pi^B_n)$, there exists a unique 
pair $\{B, B'\}$ of distinct blocks of $x$ such that 
\begin{itemize}
\itemsep=0pt
\item[{\rm (i)}]
$\min(|B|) \in B$ and $\min(|B'|) \in B'$, 

\item[{\rm (ii)}]
either $B$ and $B'$ are nonzero and $B \cup B'$ or 
$(-B) \cup B'$ is a nonzero block of $y$, or one of 
$B$ and $B'$ is the zero block of $x$ and $B \cup B'$ 
is contained in the zero block of $y$.
\end{itemize}
We then define $\lambda(x, y)$ as the maximum 
of $\min(|B|)$ and $\min(|B'|)$ and leave to the
reader to verify that the resulting map $\lambda: 
\eE(\Pi^B_n) \to \{1, 2,\dots,n\}$ is a strict 
$R$-labeling. In particular, 
$\beta_{\Pi^B_n}(S)$ is equal to the number of maximal 
chains $C$ of $\Pi^B_n$ such that $\Des_\lambda(C) = S$, 
for every $S \subseteq [n-1]$.

\medskip
\noindent
\emph{Proof of Proposition~\ref{prop:betaPiB}}.
We adapt the proof of Proposition~\ref{prop:betaPi}
as follows. We consider the multiset
\[ \bB^\ast_n \ := \ 
   \{1, 2, 2, 2, 3, 3, 3, 3, 3,\dots,n, n,\dots,n\} 
	 \times \cdots \times \{1, 2, 2, 2, 3, 3, 3, 3, 3\} 
	 \times \{1, 2, 2, 2\} \times \{1\} , \]
where the $k$th factor has $2i-1$ elements equal 
to $i$, for $i \in [n-k+1]$. We set $\Des^\ast
(\sigma) = \{ i \in [n-1]: \sigma_i > \sigma_{i+1} \}$ 
for $\sigma  = (\sigma_1, \sigma_2,\dots,\sigma_n) \in 
\bB^\ast_n$. The bijection $\bB_n \mapsto \bB^\ast_n$ 
defined by 
\[ (\sigma_1, \sigma_2,\dots,\sigma_n) \ \mapsto \ 
   (n+1-\sigma_n, n-\sigma_{n-1},\dots,2-\sigma_1) \]
shows that, for every $S \subseteq 
[n-1]$, the number of elements $\sigma \in \bB^\ast_n$ 
with $\Des^\ast(\sigma) = S$ is equal to the number 
of elements $\sigma \in \bB_n$ with $\Des(\sigma) = 
n-S$. Thus, it suffices to show that the former
equals $\beta_{\Pi^B_n}(S)$.

Let us call a nonzero block $B$ of a partition 
$x \in \Pi^B_n$ \emph{positive} if $\min(|B|) \in B$. 
Given a covering relation $(x, y) \in \eE(\Pi^B_n)$, 
we let $B_0$ be the zero block of $x$ and list its  
positive blocks, say $B_1, B_2,\dots,B_k$, so that 
$\min(|B_1|) < \min(|B_2|) < \cdots < \min(|B_k|)$. 
Then, some positive block $B_j$ merges in $y$ either
with $B_0$, or with $B_i$ or $-B_i$ for some $1 \le 
i < j$. We set $\varphi(x, y) = (j, j)$ or $(i, j)$
or $(j, i)$ in these cases, respectively, and 
$\psi(x, y) = j$. For a maximal chain $C: \hat{0} = 
c_0 \prec c_1 \prec \cdots \prec c_n = \hat{1}$ of 
$\Pi^B_n$ we set 
\begin{eqnarray*}
\widetilde{\varphi}(C) & = & (\varphi(c_0,c_1), 
\varphi(c_1,c_2),\dots,\varphi(c_{n-1},c_n)) , \\ 
\widetilde{\psi}(C) & = & (\psi(c_0,c_1), 
          \psi(c_1,c_2),\dots,\psi(c_{n-1},c_n))
\end{eqnarray*}
and consider the resulting maps 
\begin{eqnarray*}
\widetilde{\varphi} : & \mM(\Pi^B_n) \ \to & 
[n]^2 \times [n-1]^2 \times \cdots \times [1]^2 , \\ 
\widetilde{\psi} : & \mM(\Pi^B_n) \ \to & 
[n] \times [n-1] \times \cdots \times [1] .
\end{eqnarray*}

\smallskip
We note that $\widetilde{\varphi}$ is a bijection,
verify that $\Des_\lambda(C) = \Des^\ast
(\widetilde{\psi}(C))$ for every $C \in \mM(\Pi^B_n)$, 
just as in the proof of Proposition~\ref{prop:betaPi}, 
and conclude that $\beta_{\Pi^B_n}(S)$ is
equal to the number of elements $\sigma \in \bB^\ast_n$ 
with $\Des^\ast(\sigma) = S$ for every $S \subseteq 
[n-1]$. \qed

\medskip
For small values of $n$,
\[  h_{\Pi^B_n}(x) \ = \ \begin{cases}
    1, & \text{if $n = 1$} \\
    1+3x, & \text{if $n = 2$} \\
    1 + 20x + 15x^2, & \text{if $n=3$} \\
    1 + 111x + 359x^2 + 105x^3, & 
		                  \text{if $n=4$} \\
    1 + 642x + 5978x^2 + 6834x^3 + 945x^4, & 
		                  \text{if $n=5$} \\
    1 + 4081x + 92476x^2 + 268236x^3 + 143211x^4 
		         + 10395x^5, & \text{if $n=6$}.
    \end{cases}  \]
		
\begin{corollary} \label{cor:hPiB(x)}
The polynomial $h_{\Pi^B_n}(x)$ is real-rooted and 
it interlaces $h_{\Pi^B_{n+1}}(x)$ for every $n \ge 1$.
\end{corollary}

\begin{proof}
For $1 \le k \le n$ we set 
\[ h^B_{n,k}(x) \ = \sum_{\sigma \in \bB_{n,k}} 
   x^{\des(\sigma)} , \]
where $\bB_{n,k}$ is the multiset consisting of 
all words $(\sigma_1, \sigma_2,\dots,\sigma_n) \in 
\bB_n$ with $\sigma_n = k$. 
By Proposition~\ref{prop:betaPiB}, 
\[ h_{\Pi^B_n}(x) \ = \ h^B_{n+1,n+1}(x) \ = \ 
                      \sum_{k=1}^n h^B_{n,k}(x) \]
for every $n \ge 1$ and 
\[ h^B_{n+1,k}(x) \ = \ (2n-2k+3) \left( \,
   \sum_{i=1}^{k-1} h^B_{n,i}(x) \, + \, x 
	 \sum_{i=k}^n h^B_{n,i}(x) \right) \]
for $k \in [n+1]$. The result follows from these 
formulas as in the proof 
of Corollary~\ref{cor:hPi(x)}. 
\end{proof}


The intersection lattice of the Coxeter hyperplane 
arrangement of type $D_n$ is isomorphic to the 
subposet $\Pi^D_n$ of $\Pi^B_n$ which consists of 
all elements of the latter with zero block not of
the form $\{0, i, -i\}$ for any $i \in [n]$. The 
number of maximal chains of $\Pi^D_n$ can be shown 
to equal $(n!)^2/2$ for every $n \ge 2$. For small 
values of $n$,
\[  h_{\Pi^D_n}(x) \ = \ \begin{cases}
    1+x, & \text{if $n = 2$} \\
    1 + 11x + 6x^2, & \text{if $n=3$} \\
    1 + 67x + 175x^2 + 45x^3, & 
		                  \text{if $n=4$} \\
    1 + 397x + 3143x^2 + 3239x^3 + 420x^4, & 
		                  \text{if $n=5$} \\
    1 + 2539x + 50272x^2 + 134160x^3 + 67503x^4 
		         + 4725x^5, & \text{if $n=6$}.
    \end{cases} \]
		
We leave the analogue of Propositions~\ref{prop:betaPi} 
and~\ref{prop:betaPiB} for $\Pi^D_n$ open in this paper. 
Computational data suggest that the polynomial 
$h_{\Pi^D_n}(x)$ is 
real-rooted for every $n \ge 2$ and that the sequence 
$(h_{\Pi_{n+1}}(x), h_{\Pi^D_n}(x), h_{\Pi^B_n}(x))$ 
is interlacing for every $n \ge 3$.

\section{Geometric lattices and face lattices of 
polytopes} 
\label{sec:gg}

This section addresses the question of 
real-rootedness of $h_\lL(x)$ for geometric lattices, 
face lattices of polytopes and other classes
of posets. Moreover, it discusses connections among 
these questions, proves Theorem~\ref{thm:mainB} and 
deduces some partial results from that.  

We recall that a finite lattice is called geometric 
if it is atomic and semimodular (see 
\cite[Section~3.3]{StaEC1} for explanations and more 
information) or, equivalently, if it is isomorphic to 
the lattice of flats of a matroid. For a geometric 
lattice $\lL$, the possible inequalities among the 
coefficients of $h_\lL(x)$ were studied by Nyman and 
Swartz~\cite{NS04}, who showed that 
\begin{itemize}
\itemsep=0pt
\item[$\bullet$]
$a_0 \le a_1 \le \cdots 
\le a_{\lfloor (n-1)/2 \rfloor}$, and
\item[$\bullet$]
$a_i \le a_{n-1-i}$ for $0 \le i \le (n-1)/2$ 
\end{itemize}
for every geometric lattice $\lL$ of rank $n$, where
$h_\lL(x) = \sum_{i=0}^{n-1} a_i x^i$ (later, these 
inequalities were extended to the $h$-polynomials of 
$(n-1)$-dimensional simplicial complexes having a 
convex ear decomposition in~\cite{Sw06} and, more 
recently, to the $h$-polynomials of all 
$(n-1)$-dimensional doubly Cohen--Macaulay simplicial 
complexes in~\cite[Section~6]{APP21}). To the best 
of our knowledge, the unimodality of $h_\lL(x)$ is 
open. The following conjecture, which has been verified 
computationally for all geometric lattices with at 
most nine atoms, is a much stronger statement.
\begin{conjecture} 
The polynomial $h_\lL(x)$ has only real roots and
is interlaced by the Eulerian polynomial $A_n(x)$ for 
every geometric lattice $\lL$ of rank $n$.
\label{conj:geom}
\end{conjecture}

The question of real-rootedness of $h_\lL(x)$
was raised by Brenti and Welker~\cite[Question~1]{BW08} 
for face lattices of convex polytopes and (in a 
stronger form) by Athanasiadis and 
Tzanaki~\cite[Question~7.4]{AT21} for face lattices 
of more general classes of polyhedral complexes, 
including polyhedral balls and doubly Cohen--Macaulay 
polyhedral complexes. It seems natural to pose the
following even more general question. We recall that a 
finite poset $\lL$, having a minimum element $\hat{0}$, 
is said to be \emph{lower Eulerian} if the closed 
interval $[\hat{0}, x]$ in $\lL$ is Eulerian (see 
\cite[Section~3.16]{StaEC1} for the definition and 
information about Eulerian posets) for every 
$x \in \lL$.
\begin{question} 
Does $h_\lL(x)$ have only real roots for every lower 
Eulerian Cohen--Macaulay poset $\lL$?
\label{que:LECM}
\end{question}

Special classes of lower Eulerian Cohen--Macaulay 
posets for which it would be interesting to 
investigate this question include:
\begin{itemize}
\itemsep=0pt
\item[$\bullet$]
lower Eulerian Cohen--Macaulay meet semi-lattices,

\item[$\bullet$]
face posets of Cohen--Macaulay regular cell complexes,

\item[$\bullet$]
face posets of Cohen--Macaulay polyhedral complexes,

\item[$\bullet$]
Gorenstein* posets,

\item[$\bullet$]
Gorenstein* lattices,

\item[$\bullet$]
face lattices of convex polytopes 
\cite[Question~1]{BW08},

\item[$\bullet$]
face lattices of zonotopes.
\end{itemize}

\smallskip
Given Equation~(\ref{eq:hL(x)}), the following statement 
suggests that Conjecture~\ref{conj:geom} may be closely 
related to the special case of Question~\ref{que:LECM}  
concerning face lattices of zonotopes (more generally, 
of oriented matroids). We recall that $\fF(\pP)$ denotes 
the face lattice of a polytope $\pP$. For $S \subseteq 
[n-1]$, we denote by $\lpeak(S)$ the number of elements 
$i \in S$ for which $i-1 \not\in S$, known 
\cite[p.~298]{Pet15} as the \emph{left peaks} of the 
permutation $w \in \fS_n$ when $S = \Des(w)$. 
\begin{proposition} \label{prop:sd(Z)} 
For every $n$-dimensional zonotope $\zZ$ 
\begin{equation} \label{eq:hsd(Z)}
h_{\fF(\zZ)}(x) \ = \sum_{S \subseteq [n-1]} 
   \beta_\lL(S) \, (4x)^{\lpeak(S)} 
	 (1+x)^{n - 2\lpeak(S)} , 
\end{equation}	
where $\lL$ is the lattice of flats of the matroid 
associated to $\zZ$. In particular, the polynomial 
$h_{\fF(\zZ)}(x)$ has only real roots if and only if so 
does $\sum_{S \subseteq [n-1]} \beta_\lL(S) 
x^{\lpeak(S)}$.
\end{proposition}

\begin{proof}
The number of occurrences of $\ba \bb$
in a word $\ba \cdot w$ of degree $n$ in $\ba$ and 
$\bb$ is equal to $\lpeak(S)$, where $S$ is the 
subset of $[n-1]$ associated to $w$. Thus, the first 
statement follows by substituting $\ba = 1$ and 
$\bb = x$ in the formula $\Psi_{\fF(\zZ)} = \omega 
(\ba \cdot \Psi_\lL)$ of Theorem~\ref{thm:BER}. 
The second statement follows from the first and the 
fact (see \cite[Remark~3.1.1]{Ga05}) that a 
$\gamma$-positive polynomial
$\sum_{i=0}^{\lfloor n/2 \rfloor} \gamma_i x^i 
(1+x)^{n-2i}$ is real-rooted if and only if so is 
the associated $\gamma$-polynomial 
$\sum_{i=0}^{\lfloor n/2 \rfloor} \gamma_i x^i$.
\end{proof}

\begin{remark} \label{cor:sd(Z)} \rm
Let $\zZ$ and $\lL$ be as in 
Proposition~\ref{prop:sd(Z)}. Setting $x=1$ in 
Equation~(\ref{eq:hsd(Z)}) shows that the number of 
facets of the barycentric subdivision of $\zZ$ is 
equal to $2^n$ times the number of maximal chains 
of $\lL$.
\qed
\end{remark}

The main result of this section, which is a stronger 
version of Theorem~\ref{thm:mainB}, allows one to 
construct new posets with real-rooted chain
polynomials from posets known to have this property.
\begin{theorem} \label{thm:pp} 
Let $\lL$ be a bounded, graded poset of positive 
rank $n$. 
\begin{itemize}
\itemsep=0pt
\item[(a)]
We have 
\begin{eqnarray}
h_{\Pyr(\lL)}(x) & = & (1+nx) h_\lL(x) + (x - x^2) 
h^\prime_\lL(x) , \label{eq:hpyr} \\ 
h_{\Prism(\lL)}(x) & = & (1+(2n-1)x) h_\lL(x) + 
2(x - x^2) h^\prime_\lL(x) \label{eq:hprism} \\ 
& = & 2 h_{\Pyr(\lL)}(x) - (1+x) h_\lL(x) . 
\nonumber
\end{eqnarray}

\item[(b)]
Assume that $\lL$ is Cohen--Macaulay. If $h_\lL(x)$ 
is real-rooted, then $h_{\Pyr(\lL)}(x)$ and 
$h_{\Prism(\lL)}(x)$ are also real-rooted and each
of them is interlaced by $h_\lL(x)$.
\end{itemize}
\end{theorem}

\begin{proof}
Part (a) follows by substituting $\ba = 1$ and $\bb = 
x$ in the formulas of Theorem~\ref{thm:ER}. Indeed, 
if $w$ is any word of degree $n-1$ in $\ba$ and $\bb$ 
having $k$ letters equal to $\bb$, then $D(w)$ can be 
expressed as a sum of $2n-2$ words of degree $n$ in 
$\ba$ and $\bb$, of which $2k$ have $k$ letters equal 
to $\bb$ and $2n-2k-2$ have $k+1$ letters equal to 
$\bb$. As a result, Equation~(\ref{eq:ER1}) implies 
that 
\[ h_{\Pyr(\lL)}(x) \ = \ (1+x) h_\lL(x) + 
   \delta_n(h_\lL(x)) , \]
where $\delta_n: \RR_{n-1}[x] \to \RR_n[x]$ is the 
linear map defined by $\delta_n(x^k) = kx^k 
+ (n-k-1)x^{k+1}$ for $0 \le k \le n-1$. Clearly, 
$\delta_n(h(x)) = (n-1)x h(x) + (x-x^2) h^\prime(x)$ 
for every $h(x) \in \RR_{n-1}[x]$ and 
Equation~(\ref{eq:hpyr}) follows. The same argument 
shows that 
\[ h_{\Prism(\lL)}(x) \ = \ (1+x) h_\lL(x) + 2 \, 
   \delta_n(h_\lL(x)) \]
and yields Equation~(\ref{eq:hprism}).

For part (b) we rewrite Equations~(\ref{eq:hpyr}) 
and~(\ref{eq:hprism}) as 

\begin{eqnarray}
\frac{h_{\Pyr(\lL)}(x)}{(1-x)^{n+2}} & = & 
\frac{{\rm d}}{{\rm d}x} 
\left( \, \frac{x h_\lL(x)}{(1-x)^{n+1}} \right) , 
\label{eq:hpyrB} \\ & & \nonumber \\
\frac{h_\lL(x) h_{\Prism(\lL)}(x)}{(1-x)^{2n+3}} 
& = & \frac{{\rm d}}{{\rm d}x} 
\left( \, \frac{x (h_\lL(x))^2}{(1-x)^{2n+2}} 
\right) \label{eq:hprismB} .
\end{eqnarray}

\medskip
\noindent
Since $\lL$ is Cohen--Macaulay, $h_\lL(x)$ has 
nonnegative coefficients. Since the latter is 
assumed to be real-rooted and has constant term 
equal to 1, all its roots are negative. Let $d$
be the degree of $h_\lL(x)$. Then, $d \le n-1$ 
and, as it follows from part (a) and its proof,
$h_{\Pyr(\lL)}(x)$ and $h_{\Prism(\lL)}(x)$ 
have nonnegative coefficients and degree $d+1$.
Applying Rolle's theorem and taking into account 
that
\[ \lim_{x \to -\infty} 
   \frac{x h_\lL(x)}{(1-x)^{n+1}} \ = \ 
	 \lim_{x \to -\infty} 
   \frac{x (h_\lL(x))^2}{(1-x)^{2n+2}} \ = \ 0 \]
we conclude from Equations~(\ref{eq:hpyrB}) 
and~(\ref{eq:hprismB}) that each of 
$h_{\Pyr(\lL)}(x)$ and $h_{\Prism(\lL)}(x)$ has 
$d+1$ negative roots which are interlaced by those 
of $h_\lL(x)$ and the proof follows.
\end{proof}

\begin{corollary} \label{cor:matroid} 
Let $\mM$ and $\mM'$ be matroids with lattices of flats
$\lL$ and $\lL'$, respectively. If $h_\lL(x)$ is 
real-rooted and $\mM'$ is obtained by successively 
adding coloops to $\mM$, then $h_{\lL'}(x)$ is  
real-rooted as well.
\end{corollary}

\begin{proof}
This follows from part (a) of Theorem~\ref{thm:pp},
since $\lL'$ is isomorphic to $\Pyr(\lL)$
for every matroid $\mM'$ which can be obtained by 
adding one coloop to $\mM$.
\end{proof}

The following corollary of Theorem~\ref{thm:pp} 
provides classes of nonsimplicial, nonsimple and
noncubical polytopes in any dimension, the 
barycentric subdivisions of which have real-rooted 
$f$-polynomials. For instance, it implies that any 
polytope which is obtained from one of dimension at 
most 5 by applying successively the pyramid or the 
prism construction has this property.
\begin{corollary} \label{cor:polytope} 
Let $\pP$ be a convex polytope. If $h_{\fF(\pP)}
(x)$ is real-rooted, then $h_{\fF(\Pyr(\pP))}(x)$ 
and $h_{\fF(\Prism(\pP))}(x)$ have the same property
and each of them is interlaced by $h_{\fF(\pP)}(x)$.
\end{corollary}

The near-pencil of rank $n$ on $m$ elements can be 
obtained from the rank two uniform matroid on $m-n+2$
elements by adding $n-2$ coloops. Near-pencils and 
uniform matroids are known \cite[Section~3]{NS04} to 
minimize and maximize, respectively, the entries of 
the flag $h$-vector of the lattice of flats among all 
matroids of given rank and number of elements.
The following statement confirms 
Conjecture~\ref{conj:main} for the lattices of flats 
of these matroids and completes the proof of 
Theorem~\ref{thm:mainA}. Let us denote by $\lL(\mM)$ 
the lattice of flats of a matroid $\mM$. The Eulerian 
polynomial $B_n(x)$ can be defined (see, for instance,
\cite[Sections~3.1 and~3.4]{Stei92}) as $h_{\fF(\pP)}
(x)$, where $\pP$ is the $n$-dimensional cube, or as
\[ B_n(x) \ = \ \sum_{w \in \fS^\pm_n} 
   x^{\des_B(w)}, \]
where $\fS^\pm_n$ denotes (this is nonstandard notation) 
the set of signed permutations of $[n]$, meaning 
sequences $w = (w_1, w_2,\dots,w_n)$ for which $(|w_1|, 
|w_2|,\dots,|w_n|) \in \fS_n$, and $\des_B(w)$ is the 
number of indices $i \in \{0, 1,\dots,n-1\}$ such that 
$w_i > w_{i+1}$, where $w_0 := 0$. 
\begin{proposition} \label{prop:pencil-uniform} 
Let $\mM_{m,n}$ and $\uU_{m,n}$ denote the 
near-pencil and the uniform matroid, respectively, 
of rank $n$ on $m$ elements and let $\fF(\mM_{m,n})$ 
and $\fF(\uU_{m,n})$ be the face lattices of any 
zonotopes with associated matroids $\mM_{m,n}$ and 
$\uU_{m,n}$, respectively.
\begin{itemize}
\itemsep=0pt
\item[(a)]
The polynomials $h_{\lL(\mM_{m,n})}(x)$ and 
$h_{\lL(\uU_{m,n})}(x)$ have only real roots. 
Moreover, the latter is interlaced by the 
Eulerian polynomial $A_n(x)$.

\item[(b)]
The polynomials $h_{\fF(\mM_{m,n})}(x)$ and 
$h_{\fF(\uU_{m,n})}(x)$ have only real roots. 
Moreover, the latter is interlaced by the 
Eulerian polynomial $B_{n-1}(x)$.
\end{itemize}
\end{proposition}

\begin{proof}
By definition, the near-pencil $\mM_{m,n}$ is 
obtained by successively adding coloops to a 
rank two matroid. Thus, the real-rootedness of
$h_{\lL(\mM_{m,n})}(x)$ follows from 
Corollary~\ref{cor:matroid}. Similarly, since 
adding a coloop to a linear matroid $\mM$ with 
associated zonotope $\zZ$ yields a matroid whose 
associated zonotope is combinatorially isomorphic 
to the prism over $\zZ$, the real-rootedness of
$h_{\fF(\mM_{m,n})}(x)$ follows from 
Corollary~\ref{cor:polytope}.

Since the geometric lattice $\lL(\uU_{m,n})$, 
with its maximum element removed, is a simplicial 
poset with nonnegative $h$-vector, the 
real-rootedness of $h_{\lL(\uU_{m,n})}(x)$ is 
a special case of \cite[Theorem~2]{BW08}. Similarly,
since the zonotope associated to the uniform matroid
$\uU_{m,n}$ is an $n$-dimensional cubical polytope, 
the statement that $h_{\fF(\uU_{m,n})}(x)$ is 
real-rooted and interlaced by $B_{n-1}(x)$ is a 
special case of \cite[Corollary~3.5]{Ath20b+}. 
Thus, it remains to show that $h_{\lL(\uU_{m,n})}(x)$ 
is interlaced by $A_n(x)$. 

To simplify the notation,
we set $\lL := \lL(\uU_{m,n})$ and recall that 
$\overline{\lL}$ is combinatorially isomorphic to the
poset of nonempty faces of the $(n-2)$-dimensional 
skeleton, say $\Delta_{m,n-1}$, of the 
$(m-1)$-dimensional simplex $2^{[m]}$. As a 
result, $\Delta(\overline{\lL})$ is combinatorially 
isomorphic to the barycentric subdivision 
$\sd(\Delta_{m,n-1})$ and $h_\lL(x) = 
h(\sd(\Delta_{m,n-1}), x)$. As explained in 
\cite{Ath20a+} \cite[Section~1]{Ath20b+}, this 
expression implies that 
\[ h_\lL(x) \ = \ \sum_{k=0}^{n-1} c_k p_{n-1,k}(x) 
   , \]
where $h(\Delta_{m,n-1}, x) = \sum_{k=0}^{n-1} c_k 
x^k$ and $(p_{n-1,0}(x), 
p_{n-1,1}(x),\dots,p_{n-1,n-1}(x))$ is an 
interlacing sequence of real-rooted polynomials 
with nonnegative coefficients, originally defined 
in \cite{BW08}, which sum to $A_n(x)$. Since 
$\Delta_{m,n-1}$ is the one-coskeleton of the
Cohen--Macaulay simplicial complex $\Delta_{m,n}$,
as explained in the proof of 
\cite[Theorem~6.1]{AT21} we must have $1 = c_0 \le 
c_1 \le \cdots \le c_{n-1}$. An application of 
\cite[Lemma~2.2~(c)]{ABK20+} \cite[Lemma~8]{HOW99}
then shows that 
$\sum_{k=0}^{n-1} p_{n-1,k}(x) = A_n(x)$ interlaces 
$\sum_{k=0}^{n-1} c_k p_{n-1,k}(x) = h_\lL(x)$ and 
the proof follows.
\end{proof}

Computational data, along with part (b) of 
Proposition~\ref{prop:pencil-uniform}, suggest the 
following question.
\begin{question} 
Is $h_{\fF(\zZ)}(x)$ interlaced by $B_{n-1}(x)$ for 
every $n$-dimensional zonotope $\zZ$?
\label{que:zonoBn}
\end{question}

\section{An application to the second barycentric 
subdivision} 
\label{sec:bary}

As already mentioned in Section~\ref{sec:subspace}, 
the $n$th Eulerian polynomial $A_n(x)$ is equal to 
the $h$-polynomial of the first barycentric 
subdivision of the boundary complex $\partial 
\Delta_n$ of the 
$(n-1)$-dimensional simplex $\Delta_n$. As an
application of results of previous sections, we 
now give explicit combinatorial interpretations of 
the $h$-polynomial of the second barycentric 
subdivision of $\partial \Delta_n$ and of its 
associated $\gamma$-polynomial, thus answering a
question raised in \cite[Example~4.4]{Ath18}.

Let us write $\sd^2(\Delta) = \sd(\sd(\Delta))$ 
for the second barycentric subdivision of a 
simplicial complex $\Delta$. The polynomial 
$h(\sd^2(\partial \Delta_n), x)$ is an 
$x$-analogue of $(n-1)! \, n!$, which is the number 
of $(n-2)$-dimensional faces (facets) of 
$\sd^2(\Delta_n)$. For small values of $n$,

\[  h(\sd^2(\partial \Delta_n), x) \ = \ 
    \begin{cases} 
		1 + x, & \text{if $n=2$} \\
    1 + 10x + x^2, & \text{if $n=3$} \\
    1 + 71x + 71x^2 + x^3, & \text{if $n=4$} \\
    1 + 536x + 1806x^2 + 536x^3 + x^4, & 
		                         \text{if $n=5$} \\
    1 + 4677x + 38522x^2 + 38522x^3 + 4677x^4 + x^5, 
		                       & \text{if $n=6$}.
    \end{cases}  \]
		
Since $(n-1)! \, n!$ is equal to $2^{n-1}$ times 
the number of maximal chains of the partition 
lattice $\Pi_n$, and the 
latter is equal to the number of elements of the 
multiset $\aA_n$, it is not unreasonable to 
expect that the coefficients of $h(\sd^2(\partial 
\Delta_n), x)$ count signed elements of $\aA_n$
by some descent-type statistic. Indeed, let us 
denote by $\aA^\pm_n$ the multiset of all signed 
elements of $\aA_n$, meaning sequences $\tau = 
(\tau_1, \tau_2,\dots,\tau_{n-1})$ such that 
$(|\tau_1|, |\tau_2|,\dots,|\tau_{n-1}|) \in 
\aA_n$. For such $\tau \in \aA^\pm_n$, let us 
denote by $\eDes_B(\tau)$ the set of indices
$i \in \{0, 1,\dots,n-2\}$ for which
\begin{itemize}
\itemsep=0pt
\item[$\bullet$]
$\tau_i > \tau_{i+1}$, or

\item[$\bullet$]
$\tau_i = \tau_{i+1} > 0$,
\end{itemize} 
where $\tau_0 := 0$, and by $\edes_B(\tau)$ the 
cardinality of $\eDes_B(\tau)$. For example, 
the multiset $\aA^\pm_3$ consists of the twelve 
signed words $(\pm 1, \pm 1)$, $(\pm 1, \pm 1)$ 
and $(\pm 1, \pm 2)$. There is one such word 
$\tau$ with $\edes_B(\tau) = 0$, ten with 
$\edes_B(\tau) = 1$ and one with $\edes_B(\tau) 
= 2$. 

The combinatorial interpretation provided for the 
coefficients $\gamma_{n,2,i}$ in the following 
statement is analogous to the one provided by 
Petersen~\cite[Proposition~4.15]{Pet07} 
\cite[Section~13.2]{Pet15} for the coefficients
of the $\gamma$-polynomial associated to the 
Eulerian polynomial $B_n(x)$. For $\sigma \in 
\aA_n$, we denote by $\lpeak(\sigma)$ the number 
of descents $i \in [n-2]$ of $\sigma$ for which
either $i-1$ is an ascent of $\sigma$, or $i=1$.
		
\begin{proposition} \label{prop:sd^2} 
For every $n \ge 2$
\begin{eqnarray}
h(\sd^2(\partial \Delta_n), x) & = & 
\sum_{\tau \in \aA^\pm_n} x^{\edes_B(\tau)} 
\label{eq:sd^2a} \\ & = & 
\sum_{i=0}^{\lfloor (n-1)/2 \rfloor} 
\gamma_{n,2,i} \, x^i (1+x)^{n-1-2i} , 
\label{eq:sd^2b} 
\end{eqnarray}
where $\gamma_{n,2,i}$ is equal to $4^i$ times 
the number of words $\sigma \in \aA_n$ with 
$\lpeak(\sigma) = i$.
\end{proposition}

\begin{proof}
The poset of faces of $\sd(\partial \Delta_n)$
is combinatorially isomorphic 
to $\fF(\hH_n)$, where $\hH_n$ is the Coxeter 
hyperplane arrangement of type $A_{n-1}$.
As a result, we have $h(\sd^2(\partial \Delta_n), 
x) = h(\sd(\zZ_n), x) = h_{\fF(\zZ_n)}(x)$, where
$\zZ_n$ is the zonotope associated to $\hH_n$
(known as the $(n-1)$-dimensional permutohedron). 
Since the geometric lattice $\lL(\hH_n)$ is 
combinatorially isomorphic to $\Pi_n$, applying 
Proposition~\ref{prop:sd(Z)} to $\zZ_n$ we get 
\[ h(\sd^2(\partial \Delta_n), x) \ =
   \sum_{S \subseteq [n-2]} \beta_{\Pi_n} (S) \, 
	 (4x)^{\lpeak(S)} (1+x)^{n - 1 - 2\lpeak(S)} . \]
Combined with Proposition~\ref{prop:betaPi}, this 
expression yields Equation~(\ref{eq:sd^2b}). To 
deduce Equation~(\ref{eq:sd^2a}) from that, it 
suffices to show that for every $\sigma \in \aA_n$
\begin{equation} \label{eq:Pet}
\sum_{\tau \in \Sigma(\sigma)} x^{\edes_B(\tau)} 
   \ = \ (4x)^{\lpeak(\sigma)} 
	 (1+x)^{n-1-2{\lpeak(\sigma)}} ,
\end{equation}
where $\Sigma(\sigma)$ 
stands for the set of all words $(\tau_1, 
\tau_2,\dots,\tau_{n-1})$ such that $(|\tau_1|, 
|\tau_2|,\dots,|\tau_{n-1}|) = \sigma$. Indeed, given
$\sigma = (\sigma_1, \sigma_2,\dots,\sigma_{n-1}) \in 
\aA_n$ and $\tau = (\tau_1, \tau_2,\dots,\tau_{n-1}) 
\in \Sigma(\sigma)$ such that $\tau_i = \varepsilon_i 
\sigma_i$ for every $i \in [n-1]$, with $\varepsilon_i 
\in \{-1,1\}$, one can verify that

\medskip
\begin{itemize}
\itemsep=0pt
\item[$\bullet$]
$0 \in \eDes_B(\tau)$ if and only if 
$\varepsilon_1 = -1$,

\item[$\bullet$]
for $\sigma_i < \sigma_{i+1}$, we have 
$i \in \eDes_B(\tau)$ if and only if 
$\varepsilon_{i+1} = -1$,

\item[$\bullet$]
for $\sigma_i \ge \sigma_{i+1}$, we have 
$i \in \eDes_B(\tau)$ if and only if 
$\varepsilon_i = 1$.
\end{itemize} 

\medskip
\noindent
Then, the argument given for permutations $\sigma \in 
\fS_n$ in \cite[Section~13.2]{Pet15} applies verbatim
to our situation and proves~(\ref{eq:Pet}); the 
details are left to the interested reader.
\end{proof}




\medskip
\noindent \textbf{Acknowledgments}. This work was 
supported by the Hellenic Foundation for Research and 
Innovation (H.F.R.I.) under the `2nd Call for H.F.R.I. 
Research Projects to support Faculty Members \& 
Researchers'. The authors thank Eleni Tzanaki 
for her generous help with all computations (via Sage 
\cite{Sage} and the database of matroids, publicly 
available at 
{\tt 
https://www-imai.is.s.u-tokyo.ac.jp/~ymatsu/matroid/index.html}) 
mentioned in this paper and Luis Feronni for useful 
comments.

\end{document}